\documentclass[11pt,reqno]{rmmcart} \usepackage{amssymb,amsthm,amsmath}
\usepackage{enumitem}

\theoremstyle{plain}
\newtheorem{theorem}{Theorem}[section]
\newtheorem{definition/theorem}[theorem]{Definition/Theorem}
\newtheorem{cor}[theorem]{Corollary}
\newtheorem{lemma}[theorem]{Lemma}
\newtheorem{prop}[theorem]{Proposition}

\theoremstyle{definition}
\newtheorem{defn}[theorem]{Definition}

\newtheorem{example}[theorem]{Example}

\newtheorem{remark}[theorem]{Remark}
\newtheorem{Definitions and Notation}[theorem]{Definitions and
Notation}

\numberwithin{equation}{section}

\newcommand{\f}[1]{\ensuremath{\mathfrak{#1}}}

\newcommand{\ol}[1]{\ensuremath{\overline{#1}}}

\newcommand{\Z}{\ensuremath{\mathbb{Z}}}

\newcommand{\m}{\mathfrak{m}}
\newcommand{\be}{\begin{enumerate}}
\newcommand{\ee}{\end{enumerate}}
\DeclareMathOperator{\ord}{ord}
\DeclareMathOperator{\chara}{char}
\newcommand{\F}{\ensuremath{\mathbb{F}}}

\title[Sets of Lengths]{Sets of Lengths of Powers of a Variable}

\author{Richard Belshoff}

\address{Department of Mathematics, Missouri State University, Springfield, MO
  65897, USA}

\email{rbelshoff@missouristate.edu}

\author{Daniel Kline}

\address{Department of Mathematics, Milligan College, TN 37682, USA}

\email{dbkline@milligan.edu}

\author{Mark W. Rogers}

\address{Department of Mathematics, Missouri State
  University, Springfield, MO 65897, USA}

\email{markrogers@missouristate.edu}

\keywords{non-unique factorization, Artinian local ring, polynomial}

\subjclass[2000]{Primary 13A05; Secondary 13E10} 

\date{\today}

\begin{document}

\begin{abstract}
  A positive integer $k$ is a \emph{length} of a polynomial if that polynomial
  factors into a product of $k$ irreducible polynomials.  We find the set of
  lengths of polynomials of the form $x^n$ in $R[x]$, where $(R, \m)$ is an
  Artinian local ring with $\m^2=0$.
\end{abstract}

\maketitle

\section{Introduction}
%In this paper, a \emph{ring} means a commutative ring with identity.
%By a \emph{local} ring $(R,\m)$ we mean a Noetherian ring $R$ with a unique maximal ideal $\m$.
%We note that for a local ring $(R, \m)$, $R$ is Artinian if and only if $\m^n=0$ for some 
%integer $n\ge 1$. 
%This property that the maximal ideal is nilpotent
%plays an important role in what follows.
%A class of examples of Artinian local rings is given by taking 
%$R$ to be $\Z/p^n\Z=\Z_{p^n}$ where $p$ is prime
%and $n\ge 1$.  

In this paper we study the non-uniqueness of factorizations of $x^{n}$ in
$R[x]$, where $(R, \f{m})$ is a commutative Artinian local ring with identity,
with the added restriction that $\f{m}^{2} = 0$.  For example, $R$ could be the
ring $\Z/p^{2}\Z=\Z_{p^{2}}$ where $p$ is prime.  
\begin{example} Consider the following factorizations of $x^6$ in $\Z_9[x]$.
  \be
  \item $x^6 = x \cdot x \cdot x \cdot x \cdot x \cdot x$
  \item $x^6 = x \cdot x \cdot (x^{2} + 3) \cdot (x^{2} - 3)$
  \item $x^6 = (x^{2} + 3) \cdot (x^{2} + 3) \cdot (x^{2} + 3)$
  \item $x^6 = (x^{3} + 3) \cdot (x^{3} - 3)$
	\ee
\end{example}
The first factorization expresses $x^{6}$ in the usual way as a product of 6
irreducible polynomials; for this reason, we say that 6 is a \emph{length} of
$x^{6}$, and if $R$ were a unique factorization domain, this would be the only
length of $x^{6}$.  However, the remaining factorizations show that 4, 3, and 2
are lengths of $x^{6}$.  As we will later see, these are all of the lengths of
$x^{6}$, and we write $L(x^{6}) = \{2, 3, 4, 6\}$.  In general, the set of
lengths of $x^n$ in $\Z_{p^2}[x]$ depends on whether $p=2$ or $p$ is an odd
prime.  For example, in $\Z_{4}[x]$, $L(x^{6}) = \{2, 4, 6\}$.

Our goal in this paper is the collection of results
Proposition~\ref{prop:lengths<6}, Lemma~\ref{lemma:card4}, and
Theorem~\ref{thm:length of x^n}, which completely determine $L(x^{n})$ over
Artinian local rings that are not fields but for which the square of the maximal
ideal is zero.  The result depends on whether $n$ is even or odd, and whether
the cardinality of $R$ is 4 or not (there are only two such rings with
cardinality 4).  For example, if $n$ is an even integer and the cardinality of
$R$ is greater than 4, we show that
$L(x^n) = \{2, 3, 4, 5, \ldots, n-2\}\cup \{n\},$ as we saw above for $n = 6$.

For a recent survey of sets of lengths, we refer the reader to the recent paper
\cite{G} by Alfred Geroldinger.

\section{Preliminaries} \label{S:Preliminaries}

For the rest of this paper, unless otherwise specified, $(R, \f{m})$ is a
commutative Artinian local ring identity, having unique maximal ideal $\m \ne 0$
and residue field $\ol{R} = R/\f{m}$; $R[x]$ is the polynomial ring in the
variable $x$ with coefficients in $R$.  The concept of an \emph{irreducible
  element} is usually defined only for integral domains.  For rings with
zero-divisors, several different notions of irreducible have been proposed
(\cite{A1}, \cite{A2}, \cite{A3}.)  Our definition of irreducible will be the
usual one, i.e., the one that is used when $R$ \emph{is} an integral domain.  We
begin by recalling this and a few other definitions and equivalences.  Let $f(x)
= a_0 + a_1x + a_2 x^2 + \cdots + a_d x^d$ denote a polynomial of degree $d$ in
$R[x]$.
\begin{itemize}
\item The polynomial $f(x)$ is a \emph{unit} if $a_0$ is a unit and $a_i \in \m$
  for all $i>0$.
\item The polynomial $f(x)$ is a \emph{zero divisor} if each $a_i \in \m$.  Note
  that any multiple of a zero divisor is a zero divisor.
\item The polynomial $f(x)$ is \emph{regular} if $f(x)$ is not a zero divisor.
  Note that if a product is regular, so is each factor.
\item The nonunit polynomial $f(x)$ is \emph{irreducible} if $f(x)=g(x)h(x)$
  implies $g(x)$ or $h(x)$ is a unit.
\item The \emph{order} of the polynomial $f(x)$ (denoted $\ord(f)$) is the least
  $i$ such that $a_i\ne 0$.
\item By $\overline{f}(x)$ we mean the image of $f(x)$ in $\ol{R}[x]$.
\end{itemize}

The following proposition is proved by B.~R.~McDonald (\cite{M}, Theorem XIII.6)
for finite rings.  The result generalizes to the case where $R$ is any Artinian
local ring.  We will need this result in Lemma \ref{lemma:5}.

\begin{prop} \label{prop:McDonald} Every regular polynomial $f$ in $R[x]$ is
  representable as $f=uf^{*}$ where $u$ is a unit of $R[x]$ and $f^{*}$ is a
  monic polynomial of $R[x]$.  Also, $\deg(f^{*}) = \deg(\overline{f})$.
\end{prop}

The following simple corollary allows us to assume that irreducible factors of a
monic polynomial are themselves monic, and thus nonconstant.  We use this
corollary implicitly throughout the paper.

\begin{cor} \label{cor:monic} If $f$ is a monic polynomial in $R[x]$ such that
  $f = f_{1} f_{2} \cdots f_{k}$ for some polynomials $f_{1}$, $f_{2}$, \ldots,
  $f_{k}$ then there are monic polynomials $f_{1}^{*}$, $f_{2}^{*}$, \ldots,
  $f_{k}^{*}$ such that $f = f_{1}^{*} f_{2}^{*} \cdots f_{k}^{*}$.  If each
  $f_{i}$ is irreducible, then each $f_{i}^{*}$ is irreducible (and
  nonconstant).
\end{cor}  

\begin{proof}
  Since the product $f_{1} \cdots f_{k}$ is regular, so is each $f_{i}$.  By
  Proposition~\ref{prop:McDonald}, each $f_{i} = u_{i}f_{i}^{*}$ for some unit
  $u_{i}$ and some monic polynomial $f_{i}^{*}$.  Since $f = (u_{1} \cdots
  u_{k})f_{1}^{*} \cdots f_{k}^{*}$ and $f$ and $f_{1}^{*} \cdots f_{k}^{*}$ are
  both monic, the leading coefficient of the unit $u_{1} \cdots u_{k}$ is 1.
  The only unit with this property is 1, so $f = f_{1}^{*} f_{2}^{*} \cdots
  f_{k}^{*}$.  If each $f_{i}$ is irreducible, then so are the associates
  $f_{i}^{*}$; they cannot be constant, since the only monic constant is 1, and
  units aren't considered irreducible.
\end{proof}

\section{Generalized Eisenstein Polynomials}

We begin by showing that while factorization in $R[x]$ may be non-unique, it is
at least possible.  We remind the reader that $(R, \f{m})$ is an Artinian local
ring with $\f{m} \neq 0$.

\begin{prop}
  Every polynomial of positive degree in $R[x]$ that is not a unit can be
  factored into a product of irreducible polynomials.
\end{prop}

\begin{proof}
  Suppose $f(x) = a_dx^d + \cdots + a_1x + a_0$ is a zero divisor of $R[x]$.
  Then, for $0\le k\le d$, we have $a_k\in \m$, hence $a_k$ is nilpotent.  Now
  $$
  1 - f(x) = -a_dx^d - \cdots - a_1x + (1 - a_0)
  $$
  and $1 - a_0$ is a unit.  Therefore $1 - f(x)$ is a unit of $R[x]$.

  This shows that the ring $R[x]$ has \emph{harmless zero-divisors} using the
  terminology of Frei-Frisch \cite[Definition 2.3]{FF}.  Now the result follows
  from \cite[Lemma 2.8]{FF}.
\end{proof}

\begin{defn} 
  A \emph{generalized Eisenstein polynomial} (abbreviated \emph{GE~polynomial})
  is a non-constant monic polynomial $f(x) = x^{d} + f_{d - 1} x^{d - 1} +
  \cdots + f_{1} x + f_{0}$ with the property that $f_{i} \in \f{m}$ for each $i
  = 0, \ldots, d - 1$.  Equivalently, $f(x)$ is a GE~polynomial if $f(x)$ is
  non-constant, monic and $\ol{f(x)} = x^{d}$ in $\ol{R}[x]$, where $d = \deg
  f$.
\end{defn}

We note that $x^n$ is a GE~polynomial for any positive integer $n$.

\begin{lemma}\label{lemma:factor_of_GE_is_GE}
  Let $f$ and $g$ be monic, nonconstant polynomials in $R[x]$.  Then $fg$ is a
  GE~polynomial if and only if both $f$ and $g$ are GE~polynomials.
\end{lemma}

\begin{proof}
  Assume both $f$ and $g$ are GE~polynomials; then $\ol{f} = x^k$ and
  $\ol{g} = x^\ell$ where $k$ and $\ell$ are the degrees of $f$ and $g$.  Thus
  $x^{k+\ell} = \ol{f}\ol{g} = \ol{fg}$, showing that $fg$ is a GE~polynomial.

  Conversely, if $f$ and $g$ are monic of degrees $k$ and $\ell$ respectively,
  then $\ol{f}\ol{g} = \ol{fg} = x^{k+\ell}$ since $fg$ is a GE~polynomial.
  Since $\ol{R}[x]$ is a UFD, it follows easily that $\ol{f}=x^k$ and
  $\ol{g}=x^{\ell}$.  Therefore both $f$ and $g$ are GE~polynomials.
\end{proof}

The next theorem is the reason for our terminology ``generalized Eisenstein
polynomial."

\begin{theorem}\label{T:Eisenstein}
  If $f$ is a GE~polynomial in $R[x]$ whose constant term is in $\f{m} \setminus
  \f{m}^{2}$, then $f$ is irreducible.
\end{theorem}

\begin{proof}
  Suppose by way of contradiction that there are two polynomials $g, h$ with
  $f = gh$.  By Corollary~\ref{cor:monic}, $f = g^{*}h^{*}$ for some monic
  polynomials $g^{*}$, $h^{*}$.  By Lemma~\ref{lemma:factor_of_GE_is_GE}, either
  $g^{*}$ and $h^{*}$ are GE polynomials, or one of them is constant.  If one of
  them is constant then it is a unit, and the proof is complete.  If both were
  nonconstant, then since they are GE polynomials, the product of their constant
  terms would be in $\f{m}^{2}$, and this would contradict the assumption on the
  constant term of $f$.
\end{proof}  
  
\begin{cor}\label{cor:2}
  Suppose $\f{m}^{2} = 0$.  If $f$ is a GE~polynomial in $R[x]$ with degree at
  least two, then $f$ is irreducible if and only if $f$ has a nonzero constant
  term.
\end{cor}

\begin{proof}
  If $f$ is a GE~polynomial with a nonzero constant term, then the constant term
  is in $\f{m} \setminus \f{m}^{2}$ since $\f{m}^{2} = 0$.  According to
  Theorem~\ref{T:Eisenstein}, $f$ is irreducible.

  If the constant term of $f$ is zero then
  \[
    f = x^{d} + a_{d - 1} x^{d - 1} + \cdots + a_{j} x^{j} = x (x^{d - 1} + a_{d
      - 1} x^{d - 2} + \cdots + a_{j} x^{j - 1})
  \]
  where $d \geq 2$ and $1 \leq j \leq d$.  
  %If $d = j$ then $f(x) = x^{d}$ is
  %certainly reducible.  If $d > j$ then 
  The factorization displayed above is a factorization into a product of two
  non-units, since $a_{j} \in \m$.  Therefore, if the constant term of $f$ is
  zero, then $f$ is reducible.
\end{proof}

\begin{remark} \label{remark:counting} Let $(R, \m)$ be a finite local ring such
  that $\m^2=0$ and let $k = |\m|$. By Corollary \ref{cor:2}, the number of
  irreducible GE~polynomials of degree $2$ in $R[x]$ is exactly $k(k-1)$.  We
  will use this remark later in Lemma \ref{lemma:card4}.
\end{remark}

The central idea of the following proof for the case $k = 2$ was inspired by the
computations done at the start of \cite{FF}.

\begin{prop} \label{Megaprop} Suppose $\m^2=0$.  If $k\ge 2$ and $f_1, f_2,
  \ldots, f_k$ are GE polynomials in $R[x]$ with $\deg(f_i) = d_i$ and $d_1 \ge
  d_2 \ge \cdots \ge d_k$ then there is a GE polynomial $h$ of degree $d_1$ such
  that $f_1 f_2 \cdots f_k = hx^{d_2+d_3 + \cdots + d_k}$.  If, furthermore,
  $f_1$ is irreducible and $d_1 > d_2$, then $h$ is irreducible and $\ord(
  \prod_{i=1}^k f_i ) = \sum_{i=2}^k d_i$.
\end{prop}

\begin{proof}
  We use induction on $k$.  Suppose $k=2$.  We have $f_1 = x^{d_1} +
  \tilde{f_1}$ and $f_2 = x^{d_2} + \tilde{f_2}$ where $\tilde{f_{1}},
  \tilde{f_{2}} \in \f{m}[x]$ have degrees less than $d_{1}, d_{2}$,
  respectively. Therefore $f_1f_2 = (x^{d_1} + \tilde{f_1}) (x^{d_2} +
  \tilde{f_2}) = x^{d_1+d_2} + x^{d_1}\tilde{f_2} + x^{d_2}\tilde{f_1} +
  \tilde{f_1}\tilde{f_2}$. Since $\m^2=0$ we have $\tilde{f_1}\tilde{f_2} = 0$,
  so
  \begin{align*}
    f_1f_2 &= x^{d_1+d_2} + x^{d_1}\tilde{f_2} + x^{d_2}\tilde{f_1} \\
    &= (x^{d_1} + x^{d_1-d_2}\tilde{f_2} + \tilde{f_1}) x^{d_2}
  \end{align*}
  and the polynomial $h = x^{d_1} + x^{d_1-d_2}\tilde{f_2} + \tilde{f_1}$ is a
  GE polynomial of degree $d_1$.  If, furthermore, $f_{1}$ is irreducible and
  $d_{1} > d_{2}$, then $h$ is irreducible by Lemma~\ref{cor:2}, since $h$ and
  $f_{1}$ have the same constant term.  Finally, because $\tilde{f_1}$ has a
  nonzero constant term, we have $\ord(f_1f_2) = d_2$.

  Now suppose $k\ge 2$ and assume $f_1\cdots f_k = x^{d_2+\cdots +d_k}h_1$ where
  $h_1$ is a GE polynomial with $\deg(h_1) = d_1$, and if $f_{1}$ is irreducible
  with $d_{1} > d_{2}$, then $h_{1}$ is irreducible.  Then
  \begin{align*}
    \prod_{i=1}^{k+1} f_i &= f_1\cdots f_k f_{k+1} \\
    &= (h_1 x^{d_2+\cdots +d_k})f_{k+1} \\
    &= x^{d_2+\cdots +d_k}(h_1f_{k+1}) \\
    &= x^{d_2+\cdots +d_k}(h x^{d_{k+1}})\\
  \end{align*}
  for some GE polynomial $h$ of degree $d_1$ by the $k=2$ case, and if $f_{1}$
  is irreducible with $d_{1} > d_{2}$, then $h$ is irreducible. Therefore
  $\prod_{i=1}^{k+1} f_i = h x^{d_2+\cdots +d_k+ d_{k+1}}$.  Furthermore, if
  $f_{1}$ is irreducible with $d_{1}>d_{2}$, then $h$ is irreducible, and so it
  has a nonzero constant term. Therefore $\ord\left(\prod_{i=1}^{k+1} f_i\right)
  = d_2+\cdots + d_k + d_{k+1}$.  This completes the proof by induction.
\end{proof}

% \begin{prop}\label{prop:3}
%   Let $f$ and $g$ be irreducible GE~polynomials of degrees $d_{1}$ and $d_{2}$
%   in $R[x]$, where $(R,\m)$ is an Artinian local ring and $\m^2=0$. Then
%   $fg=x^{d_1+d_2}$ if and only if $g=\hat{f}$.
% \end{prop}

% \begin{proof}
%   ($\Rightarrow$) Assume $fg=x^{d_1+d_2}$. Then $x^{d_1+d_2} =
%   (x^{d_1}+\tilde{f})(x^{d_2}+\tilde{g}) = x^{d_1+d_2} + x^{d_1}\tilde{g} +
%   x^{d_2}\tilde{f}$.  Therefore $x^{d_1}\tilde{g}=-x^{d_2}\tilde{f}$.  If $d_1 =
%   d_2$, then $\tilde{g} = -\tilde{f}$ and $g=x^{d_2} + \tilde{g} = x^{d_1}
%   -\tilde{f} = \hat{f}$, as desired. If, on the other hand, $d_1 > d_2$ (without
%   loss of generality), then $fg= h x^{d_2}$ for some irreducible GE polynomial
%   $h$ by Proposition \ref{Megaprop}.  But then $h=x^{d_1}$, a contradiction.

%   ($\Leftarrow$) Follows from Lemma \ref{lemma:4}.
% \end{proof}

\section{Sets of Lengths of $x^n$}
We begin with the definition of the \emph{set of lengths} of an element.

\begin{defn} Let $R$ be a commutative Artinian local ring with identity and let
  $f \in R[x]$.  We say that a positive integer $n$ is a \textbf{length} of $f$
  if $f$ factors into a product of $n$ irreducible polynomials in $R[x]$.  We
  define the set
$$L(f) = \{ n\ |\ n\mbox{ is a length of $f$} \}$$
to be the \textbf{set of lengths of $f$}.
\end{defn}

\begin{remark}
To say that $1\in L(f)$ means precisely that $f$ is irreducible, and
in this case $L(f) = \{1\}$.  If $R$ is a unique factorization domain,
then $L(f)$ is a singleton for any polynomial $f\in R[x]$.  
Of course $n\in L(x^n)$, and if $R$ is a UFD then $L(x^n) = \{n\}$.
\end{remark}

% \begin{example}
% Continuing Example \ref{example:Z_9},
% 	$$f=x^2+3x+6 \mbox{\qquad and \qquad} \hat{f}=x^2-3x-6$$ 
% 	are irreducible polynomials in $\Z_9[x]$.
% Because $x^4 = (x^2+3x+6)(x^2-3x-6)$ we have
% $$\{2, 4\} \subseteq L(x^4).$$
% We will see soon, in Proposition \ref{prop:lengths<6}, that 
% $L(x^4) = \{2, 4\}$ in $\Z_9[x]$.
% \end{example}

A regular polynomial of degree $n$ cannot have length greater than $n$,
according to the next lemma.  In fact, after we establish the next three lemmas,
we will be able to determine the set of lengths of $x^n$ for $n\le 5$.

\begin{lemma} \label{lemma:5} If $f$ is a regular polynomial in $R[x]$ of degree
  $n$, then $L(f)\subseteq \{1, 2, \ldots, n\}$.
\end{lemma}

\begin{proof}
  If $f$ is a unit, then $L(f) = \emptyset$ since irreducibles aren't units and
  a product of nonunits can't be a unit; now assume $f$ is not a unit.  Suppose
  $k\in L(f)$; then there are irreducible polynomials $f_1, \ldots, f_k$ in
  $R[x]$ such that $f=f_1\cdots f_k$.  Each $f_{i}$ must be regular, since $f$
  is, and thus each $f_{i}$ has positive degree, since the only regular
  constants are units.  For each $i=1, \ldots, k$, we have $f_i = u_i f_i^{*}$
  for some unit $u_i$ and some monic $f_i^{*}$ in $R[x]$, by Proposition
  \ref{prop:McDonald}; since $f_{i}$ is not a unit, $f_{i}^{*}$ has positive
  degree.  We have $f=f_1\cdots f_k = u_1\cdots u_k f_1^{*} \cdots f_k^{*}$ and
  thus
  $k \le \sum_{i=1}^k \deg(f_i^{*}) \le \deg(u_1\cdots u_k f_1^{*} \cdots
  f_k^{*}) = \deg(f) = n$.
\end{proof}

The assumption that $f$ is a regular polynomial is necessary in Lemma
\ref{lemma:5}: If $R=\Z_{4}$ then the constant polynomial $2\in R[x]$ is
irreducible.  Hence for the polynomial $f=2x$ of degree $1$ we have $2\in L(f)$.

\begin{lemma}\label{lemma:neg}
  Suppose $\m^2=0$.  If $n$ is a positive integer then $n-1 \not\in L(x^n)$, and
  if $n$ is odd then $2\not\in L(x^n)$.
\end{lemma}

\begin{proof}
  We prove the second part first.  Suppose, to get a contradiction, that
  $2\in L(x^n)$ for some odd positive integer $n$.  By Lemma \ref{lemma:5} we
  must have $n\ge 3$; thus there are irreducible nonconstant monic polynomials
  $f$ and $g$ such that $x^n = fg$.  By Lemma \ref{lemma:factor_of_GE_is_GE},
  both $f$ and $g$ are GE~polynomials.  Since $n$ is odd, $\deg(f)\ne \deg(g)$,
  so without loss of generality we assume $\deg(f)>\deg(g)$.  By Proposition
  \ref{Megaprop} there is an irreducible GE~polynomial $h$ such that
  $x^n=fg = hx^{\deg(g)}$, so
  $n=\ord(x^{n}) = \ord(fg) = \deg(g) = n - \deg(f)$, contradicting $f$
  nonconstant. This shows $2\not\in L(x^n)$ if $n$ is odd.

  Now we prove the first part.  Since $x^2$ is reducible, $1\not\in L(x^2)$.  We
  have just shown $2\not\in L(x^3)$.  Suppose, to get a contradiction, $n-1\in
  L(x^n)$ for some integer $n\ge 4$; then there are irreducible, nonconstant,
  monic GE~polynomials $f_1, f_2, \ldots, f_{n-1}$ such $x^{n} = f_1 f_2 \cdots
  f_{n-1}$.  Since $\deg(f_1 f_2 \cdots f_{n-1}) = n$, exactly one $f_i$ has
  degree $2$ and the rest are linear.  Without loss of generality, assume
  polynomials $f_2$ through $f_{n-1}$ are linear and $f_1$ has degree two.  By
  Proposition \ref{Megaprop}, $f_2\cdots f_{n-1} = hx^{n-3}$ where $h$ is a
  linear GE~polynomial.  Thus $x^n = f_1 h x^{n-3}$, which implies $x^3 = f_1
  h$. This is a contradiction, since if $x^3 = f_1h$ then $2\in L(x^3)$.
  Therefore $n-1\not\in L(x^n)$
\end{proof}

%For the rest of this section, $R$ denotes a local ring with maximal 
%ideal $\m\ne 0$ and $\m^2=0$.

\begin{lemma}\label{lemma:6}
  Suppose $\m^2=0$.  Let $q$ be an irreducible GE polynomial in $R[x]$.  For any
  integer $n\ge 2$,
\begin{enumerate}
\item  If $n$ is even, then $\{2, 4, 6, \ldots, n-2, n\}\subseteq L(q^n)$.
\item  If $n$ is odd, then $\{3, 5, 7, \ldots, n-2, n\}\subseteq L(q^n)$.
\end{enumerate}
\end{lemma}
\begin{proof}
  Suppose $n$ is even.  Since $q$ is irreducible, $n\in L(q^n)$.  Let $k$ be any
  even integer such that $2\le k< n$; we will find a factorization of $q^n$ with
  length $k$.  Let $m$ be any nonzero element of the maximal ideal $\m$ and
  consider the factorization
\begin{equation}\label{eqn:1}
(q^{\frac{n-k+2}{2}}+m) 
(q^{\frac{n-k+2}{2}}-m) = q^{n-k+2} 
\end{equation}
Since $n-k\ge 2$, $\frac{n-k+2}{2}\ge 2$, hence $q^{\frac{n-k+2}{2}}$ is a
reducible GE~polynomial; by Corollary \ref{cor:2}, $q^{\frac{n-k+2}{2}}$ has
constant $0$, so $q^{\frac{n-k+2}{2}} + m$ is irreducible.  Similarly
$q^{\frac{n-k+2}{2}}-m$ is also irreducible.  Multiplying both sides of equation
(\ref{eqn:1}) by $q^{k-2}$ yields
$q^{k - 2}(q^\frac{n - k + 2}{2} + m)(q^\frac{n - k + 2}{2} - m) = q^n$.  Hence
we have a product of $k$ irreducible factors equal to $q^n$ for any even $k$
such that $2\le k < n$.  Therefore,
$\{2, 4, 6, \ldots, n-2, n\}\subseteq L(q^n)$.

If $n$ is odd, the proof follows the same argument and factorization as above,
except this time $n$ and $k$ are both odd integers.
\end{proof}

\begin{prop}\label{prop:lengths<6} Suppose $\m^2=0$.  In $R[x]$ we have

\hfill $L(x)=\{1\}$ \hfill $L(x^2)=\{2\}$ \hfill $L(x^3)=\{3\}$ 
\hfill $L(x^4)=\{2, 4\}$ \hfill $L(x^5)=\{3,5\}$ \hspace*{\fill}

%\hfill $L(x)=\{1\}$ \hfill $L(x^2)=\{2\}$ \hfill $L(x^3)=\{3\}$ \hspace*{\fill}
%\hfill $L(x^4)=\{2, 4\}$ \hfill $L(x^5)=\{3,5\}$ \hspace*{\fill}
%\begin{align*}
%L(x^n) &= \{n\} \qquad\mbox{for $1\le n\le 3$}\\
%L(x^4)&= \{2, 4\} \\
%L(x^5) &= \{3, 5\} \\
%\end{align*}
\end{prop}

\begin{proof}
  This follows directly from Lemmas \ref{lemma:5}, \ref{lemma:neg}, and
  \ref{lemma:6}.
\end{proof}

We now proceed to find the set of lengths of $x^6$.  By Lemmas \ref{lemma:5},
\ref{lemma:neg}, and \ref{lemma:6} we have
$$\{2,4,6\}\subseteq L(x^6) \subseteq \{2, 3, 4, 6\}.$$
It remains to determine if $3\in L(x^6)$; this depends on whether $|R|>4$ or
$|R|=4$ as we will see in Lemma \ref{lemma:card4} below.

We first establish some general results about local rings of
cardinality $4$.

\begin{prop} \label{prop:char2} Let $(R, \m)$ be any local ring.  The following
  are equivalent.
  \begin{enumerate}[font=\normalfont]
  \item \label{a}  $\chara(R/\m) = 2$ 
  \item \label{b}  $2\in \m$
  \end{enumerate}
  If $\m\ne 0$ but $\m^2=0$, then {\normalfont (\ref{a})} and {\normalfont
    (\ref{b})} are equivalent to
  \begin{enumerate}[resume, font=\normalfont]
  \item \label{c}  $2\m = 0$
  \end{enumerate}
\end{prop}

\begin{proof}
  We have $\chara(R/\m) = 2$ if and only if $\bar{1} + \bar{1} = \bar{2} =
  \bar{0}$ in $R/\m$ if and only if $2\in\m$.  Now assume $\m\ne 0$ but
  $\m^2=0$.  For (\ref{b}) implies (\ref{c}), given any $m\in \m$, we have
  $2m\in\m^2 = 0$.  Conversely, if $2\m = 0$ then $2$ is not a unit (since $\m
  \neq 0$), so $2\in \m$.
\end{proof}

\begin{prop} \label{prop:m2} Let $(R,\m)$ be any local ring.  If $|\m|=2$, then
  $|R|=4$.
\end{prop}

\begin{proof} Since $R = R^\times \cup \m$, the disjoint union of the units
  $R^\times$ and the maximal ideal $\m$, it suffices to show that $R$ has
  exactly two units.  Suppose $\m = \{0, t\}$. If the only unit of $R$ is $1$,
  then $R$ is a ring with three elements and is thus isomorphic to $\Z_3$,
  contradicting $|\m| = 2$.  Therefore there exists a unit $u\ne 1$ in $R$; we
  show $u = t + 1$.  Since $ut\in \m$, either $ut=0$ or $ut=t$.  The first case
  is impossible since $t \neq 0$.  In the second case $t(u-1) = 0$ and hence
  $u-1\in\m$.  This implies $u-1=t$ so $u=t+1$.  This shows that $R=\{0, 1, t,
  t+1\}$, a ring with four elements.
\end{proof}

\begin{remark} \label{remark:4_elements} It is known (\cite[Exercise I.4,
  p.4]{M}) that if $R$ is any ring with four elements, then $R$ must be
  isomorphic to one of the following: $\Z_4$, $\F_4$, $\Z_2\times\Z_2$, or
  $\F_2[t]/(t^2)$.  Of these, the only ones that are local rings and are not
  fields are $\Z_4$ and $\F_2[t]/(t^2)$.  Note that both of these have the
  equivalent properties (1), (2), (3) of Proposition \ref{prop:char2}.  Also
  note that if $R=\Z_4$ or $R=\F_2[t]/(t^2)$ there are exactly two irreducible
  GE~polynomials of degree $2$ in $R[x]$. (See Remark \ref{remark:counting}.)
  We will need this fact in the next proof.
\end{remark}

In the next lemma, we find the set of lengths of $x^{6}$; it will also be used
as the base for an induction in the proposition to follow.

\begin{lemma} \label{lemma:card4} Suppose $\f{m}^{2} = 0$.  If $|R| > 4$ then
  $L(x^6) = \{2,3,4,6\}$; if $|R| = 4$ then $L(x^6) = \{2,4,6\}$.
\end{lemma}

\begin{proof}
  By the remarks after Proposition \ref{prop:lengths<6}, it is enough to show
  that (a) if $|R|>4$ then $3\in L(x^6)$, and (b) if $|R|=4$ then
  $3\not\in L(x^6)$.

  Proof of (a):  We first show there exist three nonzero elements $a, b, c$ in
  $\m$ satisfying $a+b+c=0$.

  Suppose $\chara(R/\m) = 2$.  By Proposition \ref{prop:m2}, we know there are
  two distinct nonzero elements $a$ and $b$ in $\m$.  We must have $a+b\ne 0$,
  since otherwise $a=-b = b$ by Proposition \ref{prop:char2} ($3$), a
  contradiction.  With $c=-(a+b)$ we have $a+b+c=0$ for three nonzero elements
  $a, b, c$.

  If, on the other hand, $\chara(R/\m) \ne 2$, then by Proposition
  \ref{prop:char2}, there exists a nonzero element $a\in\m$ with $2a\ne 0$.  Now
  set $b=a$ and $c=-2a$.  Then $a+b+c=0$ and all three elements are nonzero.

  Now by Corollary \ref{cor:2}, each of the polynomials $x^2+a$, $x^2+b$, and
  $x^2+c$ is an irreducible GE polynomial.  Since $a+b+c=0$ and $\m^2=0$, the
  factorization $(x^2+a)(x^2+b)(x^2+c) = x^6$ shows that $3\in L(x^6)$.

  Proof of (b): Suppose $|R|=4$ and $3\in L(x^6)$.  Then there exists three
  irreducible, monic, nonconstant GE~polynomials $f_1, f_2$, $f_3$ whose product
  is $x^6$.  Without loss of generality we have the following three cases for
  $(\deg(f_{1}), \deg(f_{2}), \deg(f_{3}))$:  $(4, 1, 1)$, $(3, 2, 1)$, and
  $(2, 2, 2)$.  For the first two cases, $\deg(f_1)$ is greater than $\deg(f_2)$
  and $\deg(f_3)$, so by Proposition \ref{Megaprop},
  $6 = \ord(x^6) = \ord(f_1 f_2 f_3) = \deg(f_2)+\deg(f_3) < 6$, which is a
  contradiction.

  For the last case, since, as noted in Remark~\ref{remark:counting}, there are
  exactly two irreducible GE~polynomials of degree $2$ in $R[x]$, at least two
  $f_i$ are the same, say $f_1 = f_2$, and thus, since $\f{m}^{2} = 0$,
  $f_1f_2=x^4$.  But since $f_1 f_2 f_3 = x^6$, we have $f_3=x^2$, a
  contradiction since $f_3$ is irreducible.  So $3\not\in L(x^6)$.
\end{proof}  

\begin{prop}\label{prop:n-3}
  Suppose $\f{m}^{2} = 0$.  For all $n \geq 6$, $|R|>4$ if and only if
  $n-3 \in L(x^n)$.
\end{prop}

\begin{proof}
  If $|R| >4$ then by Lemma~\ref{lemma:card4}, $3\in L(x^6)$, so there is
  a factorization of $x^6$ into three irreducible polynomials. Multiplying this
  factorization by $x^{n-6}$ gives a factorization of $x^n$ of length
  $n-3$. Therefore $n-3\in L(x^n)$.

  Now assume $|R| = 4$.  We show $n-3\not\in L(x^n)$ for $n\ge 6$ (equivalently,
  $n \not\in L(x^{n + 3})$ for $n \ge 3$) by induction on $n$.  By Lemma
  \ref{lemma:card4}, $3\not\in L(x^6)$.  Now assume $k\not\in L(x^{k+3})$ for
  some $k\ge 3$.  We show $k+1\not\in L(x^{k+4})$.

  Suppose by way of contradiction that $k+1\in L(x^{k+4})$; then there exist
  $k+1$ irreducible, monic, nonconstant GE~polynomials
  $f_1, f_2, \ldots, f_{k+1}$, whose product is $x^{k+4}$.  At least one $f_i$
  must be linear, since otherwise
  $k + 4 = \deg(x^{k + 4}) = \sum_{i = 1}^{k + 1} \deg(f_{i}) \geq 2(k + 1)$,
  which is impossible since $k \geq 3$. Furthermore, at least one $f_i$ must be
  non-linear.  Without loss of generality, let $f_1$ be linear and $f_{k+1}$ be
  non-linear.  Then by Proposition \ref{Megaprop} there exists an irreducible
  GE~polynomial $h$ such that $f_{k+1}f_1 = hx$. Therefore
  $f_{1} \cdots f_{k} = (f_{k+1} f_1)f_{2} \cdots f_{k} = (hx)f_{2} \cdots
  f_{k}$.  We now have $hf_{2} \cdots f_{k} = x^{k+3}$ which implies
  $k\in L(x^{k+3})$.  This contradicts our assumption.  Therefore
  $n\not\in L(x^{n+3})$ for $n\ge 3$, or equivalently, $n-3\not\in L(x^n)$.
\end{proof}

The next two Lemmas do not depend on the cardinality of the local ring $R$.

\begin{lemma}\label{lemma:7}
  Suppose $\f{m}^{2} = 0$.  For all $n\ge 7$, $3\in L(x^n)$
\end{lemma}

\begin{proof}
  Let $n\ge 7$.  Let $m$ be a nonzero element of the maximal ideal $\m$, let
  $\ell$ be a positive integer, and consider the following three factorizations.
\begin{equation}\label{eqn:a}
(x^\ell + m)(x^{\ell + 1} - m)(x^{\ell + 1} - mx + m) = x^{3\ell + 2}
\end{equation}
\begin{equation}\label{eqn:b}
(x^\ell + m)(x^{\ell + 2} - m)(x^{\ell + 2} - mx^2 + m) = x^{3\ell + 4}
\end{equation}
\begin{equation}\label{eqn:c}
(x^\ell + m)(x^{\ell + 3} - m)(x^{\ell + 3} - mx^3 + m) = x^{3\ell + 6}
\end{equation}
By Corollary \ref{cor:2}, each polynomial on the left side of the three
factorizations is irreducible.  Assume $n\ge 7$.  There are three cases:
\begin{description}
\item[$n \equiv 0 \pmod{3}$] Set $\ell = \frac{n-6}{3}$; then from equation
  (\ref{eqn:c}), $3\in L(x^n)$.
\item[$n \equiv 1 \pmod{3}$] Set $\ell = \frac{n-4}{3}$; then from equation
  (\ref{eqn:b}), $3\in L(x^n)$.
\item[$n \equiv 2 \pmod{3}$] Set $\ell = \frac{n-2}{3}$; then from equation
  (\ref{eqn:a}), $3\in L(x^n)$.
\end{description}
Therefore for any integer $n\ge 7$, we have $3\in L(x^n)$.
\end{proof}
 
\begin{lemma}\label{lemma:new}
  Suppose $\f{m}^{2} = 0$.  For any integer $n\ge 7$:  \be[font=\normalfont]
\item $\{3,4, 5,\ldots, n-4\}\cup\{n-2, n\}\subseteq L(x^n)$ if $n$ is odd.
\item $\{2,3,4,5,\ldots, n-4\}\cup\{n-2, n\}\subseteq L(x^n)$ if $n$ is even.
  \ee
\end{lemma}

\begin{proof}
  Suppose $n\ge 7$ and $n$ is odd.  If $n=7$, then $\{3, 5, 7\}\subseteq L(x^7)$
  by Lemma~\ref{lemma:6}. Thus we may assume $n\ge 9$.  Again by
  Lemma~\ref{lemma:6}, $\{3, 5, 7,\cdots, n-2, n\}\subseteq L(x^n)$, so it
  remains to show that if $k$ is even and $4\le k\le n-4$, then $k\in L(x^n)$.
  If $4\le k \le n-4$, then $n-k+3\ge 7$, so by Lemma~\ref{lemma:7},
  $3\in L(x^{n-k+3})$; that is, there exists a factorization of $x^{n-k+3}$ of
  length $3$.  Multiplying both sides of this factorization by $x^{k-3}$ we have
  $k\in L(x^n)$. This proves (1).

  Now suppose $n$ is even and $n\ge 8$.  By Lemma~\ref{lemma:6},
  $\{2, 4, 6, \ldots, n-2, n\}\subseteq L(x^n)$.  It remains to show that if $k$
  is odd and $3\le k\le n-4$, then $k\in L(x^n)$.  If $3\le k \le n-4$ then
  $n-k+3\ge 7$.  By Lemma~\ref{lemma:7}, $3\in L(x^{n-k+3})$. That is, there is
  a factorization of $x^{n-k+3}$ into three irreducible polynomials.
  Multiplying both sides of this factorization by $x^{k-3}$ show $k\in L(x^n)$.
\end{proof}

The following is the main result of this paper, along with
Proposition~\ref{prop:lengths<6} and Lemma~\ref{lemma:card4}.  Note that the
only difference the cardinality of $R$ makes is in whether or not
$n - 3 \in L(x^{n})$.

\begin{theorem}\label{thm:length of x^n}
  Suppose $\f{m}^{2} = 0$.  Let $n$ be an integer with $n \geq 7$.\\
  If $|R|>4$ then
  \begin{align*}
    L(x^n) &=\{3, 4, 5, \ldots, n-2\}\cup \{n\} \text{ if $n$ is odd, and}\\
    L(x^n) &=\{2, 3, 4, 5, \ldots, n-2\}\cup \{n\} \text{ if $n$ is even.}
  \end{align*}
  If $|R| = 4$ then
  \begin{align*}
    L(x^n) &= \{3, 4, 5, \ldots, n-4\}\cup \{n-2, n\} \text{ if $n$ is odd, and}\\
    L(x^n) &= \{2, 3, 4, 5, \ldots, n-4\}\cup \{n-2, n\} \text{ if $n$ is even.}
  \end{align*}
\end{theorem}

\begin{proof}
  Let $n\ge 7$.  Regardless of the cardinality of $R$, by Lemmas~\ref{lemma:5}
  and \ref{lemma:new}, if $n$ is odd,
  \[
    \{3, 4, \ldots, n-4\}\cup \{n-2,n\} \subseteq L(x^n) \subseteq\{1, 2,
    \ldots,n\},
  \]
  and if $n$ is even, 
  \[
    \{2, 3, 4, \ldots, n-4\}\cup \{n-2,n\} \subseteq L(x^n) \subseteq\{1, 2,
    \ldots,n\}.
  \]
  Since $x^n$ is reducible, $1\not\in L(x^n)$.  By Lemma~\ref{lemma:neg},
  $n-1\not\in L(x^n)$, and $2\not\in L(x^n)$ if $n$ is odd.  Finally, by
  Proposition~\ref{prop:n-3}, $n-3\in L(x^n)$ if and only if $|R|>4$.
\end{proof}

%\section{The case $|R|=4$}

% The following theorem, which only applies to two rings, is a companion to our
% main result, Theorem \ref{thm:length of x^n}.  Note that the only difference is
% that in the following $n-3\not\in L(x^n)$.

% \begin{theorem}\label{thm:case_card4}
%   Assume $|R|=4$.  In $R[x]$ we have the following for any integer $n\ge 7$.
%   \begin{align*}
%     L(x^n) &= \{3, 4, 5, \ldots, n-4\}\cup \{n-2, n\} \text{ if $n$ is odd,}\\
%     L(x^n) &= \{2, 3, 4, 5, \ldots, n-4\}\cup \{n-2, n\} \text{ if $n$ is even.}
%   \end{align*}
% \end{theorem}

% \begin{proof}
% Suppose $n\ge 7$ and $n$ is odd.
% By Lemmas~\ref{lemma:5} and \ref{lemma:new},
% $$\{3, 4, \ldots, n-4\}\cup \{n-2,n\} \subseteq L(x^n) 
% \subseteq\{1, 2, \ldots, n\}.$$
% Similarly if $n\ge 8$ and $n$ is even, 
% $$\{2, 3, 4, \ldots, n-4\}\cup \{n-2,n\} \subseteq L(x^n) 
% \subseteq\{1, 2, \ldots, n\}.$$
% Since $x^n$ is reducible, $1\not\in L(x^n)$.
% By Lemma~\ref{lemma:neg}, 
% $n-1\not\in L(x^n)$, and 
% $2\not\in L(x^n)$ if $n$ is odd.
% Finally, by Lemma~\ref{lemma:|R|=4}, $n-3\not\in L(x^n)$.
% \end{proof}

\begin{example}
In $\Z_{p^2}[x]$ where $p$ is prime, we have
%$L(x^5) = \{3, 5 \}$, and 
\begin{align*}
L(x^{10}) &= \{2, 3, 4, 5, 6, 7, 8,10 \} \quad\mbox{ if $p>2$,} \\
L(x^{10}) &= \{2, 3, 4, 5, 6, 8,10 \} \quad\quad\mbox{ if $p=2$.}
\end{align*}
\end{example}

\bigskip

\noindent {\bf Acknowledgment:} This article is a generalization of part of
Daniel Kline's Master's Thesis \cite{K} under the direction of Mark Rogers.  The
thesis work was inspired by the paper \cite{FF} of Sophie Frisch and Christopher
Frei, and a private email exchange with Sophie Frisch.  The authors wish to
thank the referee for several helpful suggestions and corrections.


\begin{thebibliography}{Mat}

\bibitem[A1]{A1} Anderson,~D.~D.~and Valdes-Leon, Silvia, {\em Factorization in
    commutative rings with zero divisors}, Rocky Mountain J.~Math 26 (1996),
  439--480.

\bibitem[A2]{A2} Anderson,~D.~D.~and Valdes-Leon, Silvia, {\em Factorization in
    commutative rings with zero divisors II}, Lecture Notes in Pure and
  Appl.~Math, vol.~189, 197--219.  Dekker, New York, 1997.

\bibitem[A3]{A3} A{\=g}arg{\"u}n, Ahmet~G.~and Anderson,~D.~D.~and Valdes-Leon,
  Silvia, {\em Factorization in commutative rings with zero divisors III}, Rocky
  Mountain J.~Math 31 (2001), 1--21.

\bibitem[FF]{FF} Frei,~C.~and Frisch,~S., {\em Non-unique factorization of
    polynomials over residue class rings of the integers}, Comm.~Algebra 39
  (2011), no. 4, 1482-1490.

\bibitem[G]{G} Geroldinger, A., \emph{Sets of Lengths}, arXiv:1509.07462 [math.GR]
  
\bibitem[K]{K} Kline, D., \emph{Sets of Lengths over Residue Class Rings of the
    Integers}, Master's thesis, Missouri State University, 2011.
  
\bibitem[M]{M} McDonald, Bernard, {\em Finite Rings with Identity}, Pure and
  Applied Mathematics, Vol. {\bf 28}, Marcel Dekker, New York, 1974

 \end{thebibliography}
\end{document}